\documentclass[12pt]{article}
\usepackage{amsmath}
\usepackage{amsthm}
\usepackage{amsfonts}
\usepackage{eucal}
\newtheorem{lem}{Lemma}

\newtheorem{conj}{Conjecture}

\def\zt{\zeta}

\def\lfl{\lfloor}
\def\rfl{\rfloor}
\begin{document}
\title{On Louchard's Asymptotic Series}
\author{Michael E. Hoffman\\
\small Dept. of Mathematics, U. S. Naval Academy\\
\small Annapolis, MD 21402 USA\\
\small {\tt meh@usna.edu}}
\date{October 31, 2017}
\maketitle
\begin{abstract}
Recently G. Louchard obtained an asymptotic series $\sum_{j=0}^\infty
\frac{I_j}{n^j}$ for the integral $\int_0^1[x^n+(1-x)^n]^{\frac1{n}}dx$
as $n\to\infty$, and computed $I_j$ for $j\le 5$ in terms of 
values of the Riemann zeta function.  An interesting feature of the 
computation is that the $I_j$ are first obtained in terms of alternating 
multiple zeta values, but then everything except products of ordinary
zeta values cancels out.  We obtain similar formulas for $I_n$,
$6\le n\le 9$, and conjecture a general formula for $I_n$ in terms of 
alternating multiple zeta values.
We also conjecture that $I_n$ is a rational polynomial in the ordinary
zeta values.
\end{abstract}
\par\noindent
Louchard \cite{L} defines 
\begin{equation}
\label{idef}
I(n)=\int_0^1[x^n+(1-x)^n]^{\frac1{n}}dx
\end{equation}
and obtains an asymptotic series
\[
I(n)=I_0+\frac{I_1}{n}+\frac{I_2}{n^2}+\frac{I_3}{n^3}+\cdots
\]
as follows.
Because of the symmetry around $x=\frac12$ in (\ref{idef}), one can
write
\[
I(n)=2\int_0^{\frac12}[x^n+(1-x)^n]^{\frac1n}dx=2\int_0^{\frac12}(1-x)
\left[1+\left(\frac{x}{1-x}\right)^n\right]^{\frac1{n}}dx .
\]
Louchard makes the change of variable $x=\frac12-\frac{u}{2n}$ to get
\begin{multline*}
\frac1{2n}\int_0^{2n}\left(\frac12+\frac{u}{4n}\right)\left[
1+\left(\frac{1-\frac{u}{2n}}{1+\frac{u}{2n}}\right)^n\right]^{\frac1n}du\\
=\frac1{2n}\int_0^{2n}\left(\frac12+\frac{u}{4n}\right)\left[
1+\exp\left(\log\left(\frac{1-\frac{u}{2n}}{1+\frac{u}{2n}}
\right)^n\right)\right]^{\frac1n}du\\
=\frac1{2n}\int_0^{2n}\left(\frac12+\frac{u}{4n}\right)\left[
1+\exp\left(n\log\left(\frac{1-\frac{u}{2n}}{1+\frac{u}{2n}}\right)
\right)\right]^{\frac1n}du\\
=\frac1{2n}\int_0^{2n}\left(\frac12+\frac{u}{4n}\right)\left[1+
\exp\left(-u-\frac{u^3}{12n^2}-\frac{u^5}{80n^4}-\frac{u^7}{448n^6}
-\cdots\right)\right]^{\frac1n}du\\
=\frac1{2n}\int_0^{2n}\left(\frac12+\frac{u}{4n}\right)\left[1+
e^{-u}\exp\left(-\frac{u^3}{12n^2}-\frac{u^5}{80n^4}-\frac{u^7}{448n^6}
-\cdots\right)\right]^{\frac1n}du\\
=\frac1{2n}\int_0^{2n}\left(\frac12+\frac{u}{4n}\right)\bigg[1+e^{-u}
-\frac{e^{-u}u^3}{12n^2}-\frac{e^{-u}u^5}{80n^4}+\frac{e^{-u}u^6}{288n^4}
-\frac{e^{-u}u^7}{448n^6}+\frac{e^{-u}u^8}{960n^6}\\
-\frac{e^{-u}u^9}{10368n^6}+\cdots\bigg]^{\frac1n}du\\
=\frac1{2n}\int_0^{2n}\left(\frac12+\frac{u}{4n}\right)\exp\bigg[\frac1n\log
\bigg(1+e^{-u}-\frac{e^{-u}u^3}{12n^2}-\frac{e^{-u}u^5}{80n^4}
+\frac{e^{-u}u^6}{288n^4}-\frac{e^{-u}u^7}{448n^6}\\
+\frac{e^{-u}u^8}{960n^6}-\frac{e^{-u}u^9}{10368n^6}
-\frac{e^{-u}u^9}{2304n^8}+\frac{71e^{-u}u^{10}}{268800n^8}
-\frac{e^{-u}u^{11}}{23040n^8}+\frac{e^{-u}u^{12}}{497664n^8}+\cdots\bigg)\bigg]du\\
=\frac1{2n}\int_0^{2n}\left(\frac12+\frac{u}{4n}\right)\exp\bigg[\frac1n\log
(1+e^{-u})+\frac1n\log\bigg(1-\frac{e^{-u}u^3}{12n^2(1+e^{-u})}
-\frac{e^{-u}u^5}{80n^4(1+e^{-u})}\\
+\frac{e^{-u}u^6}{288n^4(1+e^{-u})}
-\frac{e^{-u}u^7}{448n^6(1+e^{-u})}
+\frac{e^{-u}u^8}{960n^6(1+e^{-u})}-\frac{e^{-u}u^9}{10368n^6(1+e^{-u})}\\
-\frac{e^{-u}u^9}{2304n^8(1+e^{-u})}+\frac{71e^{-u}u^{10}}{268800n^8(1+e^{-u})}
-\frac{e^{-u}u^{11}}{23040n^8(1+e^{-u})}+\frac{e^{-u}u^{12}}{497664n^8(1+e^{-u})}
+\cdots\bigg)\bigg]du\\
=\frac1{2n}\int_0^{2n}\left(\frac12+\frac{u}{4n}\right)\exp\bigg[
\frac{\log(1+e^{-u})}{n}-\frac{e^{-u}u^3}{12n^3(1+e^{-u})}
-\frac{e^{-u}u^5}{80n^5(1+e^{-u})}
+\frac{e^{-u}u^6}{288n^5(1+e^{-u})^2}\\
-\frac{e^{-u}u^7}{448n^7(1+e^{-u})}
+\frac{e^{-u}u^8}{960n^7(1+e^{-u})^2}
+\frac{e^{-u}(e^{-u}-1)u^9}{10368n^7(1+e^{-u})^3}
+O\left(\frac1{n^9}\right)\bigg]du .
\end{multline*}
Now expand the exponential in series:
\begin{multline*}
\frac1{2n}\int_0^{2n}\left(\frac12+\frac{u}{4n}\right)\bigg[1+
\frac{\log(1+e^{-u})}{n}-\frac{e^{-u}u^3}{12n^3(1+e^{-u})}
-\frac{e^{-u}u^5}{80n^5(1+e^{-u})}\\
+\frac{e^{-u}u^6}{288n^5(1+e^{-u})^2}
-\frac{e^{-u}u^7}{448n^7(1+e^{-u})}+\frac{e^{-u}u^8}{960n^7(1+e^{-u})^2}
-\frac{e^{-u}(1-e^{-u})u^9}{10368n^7(1+e^{-u})^3}\\
+\frac{\log(1+e^{-u})^2}{2n^2}+\frac{e^{-2u}u^6}{288n^6(1+e^{-u})^2}
-\frac{\log(1+e^{-u})e^{-u}u^3}{12n^4(1+e^{-u})}
-\frac{\log(1+e^{-u})e^{-u}u^5}{80n^6(1+e^{-u})}\\
+\frac{\log(1+e^{-u})e^{-u}u^6}{288n^6(1+e^{-u})^2}
-\frac{\log(1+e^{-u})e^{-u}u^7}{448n^8(1+e^{-u})}
+\frac{\log(1+e^{-u})e^{-u}u^8}{960n^8(1+e^{-u})^2}\\
+\frac{\log(1+e^{-u})e^{-u}(e^{-u}-1)u^9}{10368n^8(1+e^{-u})^3}
+\frac{e^{-2u}u^8}{960n^8(1+e^{-u})^2}
-\frac{e^{-2u}u^9}{3456n^8(1+e^{-u})^3}
+\frac{\log(1+e^{-u})^3}{6n^3}\\
-\frac{\log(1+e^{-u})^2e^{-u}u^3}{24n^5(1+e^{-u})}
-\frac{\log(1+e^{-u})^2e^{-u}u^5}{160n^7(1+e^{-u})}
+\frac{\log(1+e^{-u})^2e^{-u}u^6}{576n^7(1+e^{-u})^2}
+\frac{\log(1+e^{-u})e^{-2u}u^6}{288n^7(1+e^{-u})^2}\\
+\frac{\log(1+e^{-u})^4}{24n^4}
-\frac{\log(1+e^{-u})^3e^{-u}u^3}{72n^6(1+e^{-u})}
-\frac{\log(1+e^{-u})^3e^{-u}u^5}{480n^8(1+e^{-u})}
+\frac{\log(1+e^{-u})^3e^{-u}u^6}{1728n^8(1+e^{-u})^2}\\
+\frac{\log(1+e^{-u})^2e^{-2u}u^6}{576n^8(1+e^{-u})^2}
+\frac{\log(1+e^{-u})^5}{120n^5}
-\frac{\log(1+e^{-u})^4e^{-u}u^3}{288n^7(1+e^{-u})}
+\frac{\log(1+e^{-u})^6}{720n^6}\\
-\frac{\log(1+e^{-u})^5e^{-u}u^3}{1440n^8(1+e^{-u})}
+\frac{\log(1+e^{-u})^7}{5040n^7}
+\frac{\log(1+e^{-u})^8}{40320n^8}+O\left(\frac1{n^9}\right)\bigg]du .
\end{multline*}
Separate the first term as
\[
\frac1{2n}\int_0^{2n}\left(\frac12+\frac{u}{4n}\right)du
=\frac1{2n}\left(n+\frac{4n^2}{8n}\right)=\frac34,
\]
and write the remainder in powers of $\frac1{n}$ as
\begin{multline*}
\int_0^{2n}\bigg[\frac1{4n^2}\log(1+e^{-u})+\frac1{8n^3}
\left(u\log(1+e^{-u})+\log(1+e^{-u})^2\right)+\\
\frac1{48n^4}\left(-\frac{u^3e^{-u}}{1+e^{-u}}+3u\log(1+e^{-u})^2
+2\log(1+e^{-u})^3\right)+\\
\frac1{96n^5}\left(-\frac{u^4e^{-u}}{1+e^{-u}}+2u\log(1+e^{-u})^3
-\frac{2u^3e^{-u}\log(1+e^{-u})}{1+e^{-u}}+\log(1+e^{-u})^4\right)\\
+\frac1{5760n^6}\bigg(\frac{5u^6e^{-u}}{(1+e^{-u})^2}
-\frac{18u^5e^{-u}}{1+e^{-u}}-\frac{60u^4e^{-u}\log(1+e^{-u})}{1+e^{-u}}
-\frac{60u^3e^{-u}\log(1+e^{-u})^2}{1+e^{-u}}\\
+30u\log(1+e^{-u})^4+12\log(1+e^{-u})^5\bigg)
+\frac1{11520n^7}\bigg(\frac{10u^6e^{-2u}}{(1+e^{-u})^2}
+\frac{5u^7e^{-u}}{(1+e^{-u})^2}\\
+\frac{10u^6e^{-u}\log(1+e^{-u})}{(1+e^{-u})^2}
-\frac{36u^5e^{-u}\log(1+e^{-u})}{1+e^{-u}}
-\frac{40u^3e^{-u}\log(1+e^{-u})^3}{1+e^{-u}}
-\frac{18u^6e^{-u}}{1+e^{-u}}\\
-\frac{60u^4e^{-u}\log(1+e^{-u})^2}{1+e^{-u}}
+12u\log(1+e^{-u})^5+4\log(1+e^{-u})^6\bigg)\\
+\frac1{1451520n^8}\bigg(-\frac{35u^9e^{-u}(1-e^{-u})}{(1+e^{-u})^3}
+\frac{378u^8e^{-u}}{(1+e^{-u})^2}
+\frac{630u^7e^{-u}\log(1+e^{-u})}{(1+e^{-u})^2}\\
+\frac{630u^6e^{-u}\log(1+e^{-u})^2}{(1+e^{-u})^2}
+\frac{630u^7e^{-2u}}{(1+e^{-u})^2}
+\frac{1260u^6e^{-2u}\log(1+e^{-u})}{(1+e^{-u})^2}
-\frac{810u^7e^{-u}}{1+e^{-u}}\\
-\frac{2268u^6e^{-u}\log(1+e^{-u})}{1+e^{-u}}
-\frac{2268u^5e^{-u}\log(1+e^{-u})^2}{1+e^{-u}}
-\frac{2520u^4e^{-u}\log(1+e^{-u})^3}{1+e^{-u}}\\
-\frac{1260u^3e^{-u}\log(1+e^{-u})^4}{1+e^{-u}}
+252u\log(1+e^{-u})^6+72\log(1+e^{-u})^7\bigg)\\
+\frac1{2903040n^9}\bigg(-\frac{35u^{10}e^{-u}(1-e^{-u})}{(1+e^{-u})^3}
-\frac{70u^9e^{-u}(1-e^{-u})\log(1+e^{-u})}{(1+e^{-u})^3}
-\frac{210u^9e^{-2u}}{(1+e^{-u})^3}\\
+\frac{756u^8e^{-2u}}{(1+e^{-u})^2}
+\frac{1260u^7e^{-2u}\log(1+e^{-u})}{(1+e^{-u})^2}
+\frac{1260u^6e^{-2u}\log(1+e^{-u})^2}{(1+e^{-u})^2}
+\frac{378u^9e^{-u}}{(1+e^{-u})^2}\\
+\frac{756u^8e^{-u}\log(1+e^{-u})}{(1+e^{-u})^2}
+\frac{630u^7e^{-u}\log(1+e^{-u})^2}{(1+e^{-u})^2}
+\frac{420u^6e^{-u}\log(1+e^{-u})^3}{(1+e^{-u})^2}\\
-\frac{810u^8e^{-u}}{1+e^{-u}}-\frac{1620u^7e^{-u}\log(1+e^{-u})}{1+e^{-u}}
-\frac{2268u^6e^{-u}\log(1+e^{-u})^2}{1+e^{-u}}\\
-\frac{1512u^5e^{-u}\log(1+e^{-u})^3}{1+e^{-u}}
-\frac{1260u^4e^{-u}\log(1+e^{-u})^4}{1+e^{-u}}
-\frac{504u^3e^{-u}\log(1+e^{-u})^5}{1+e^{-u}}\\
+72u\log(1+e^{-u})^7+18\log(1+e^{-u})^8\bigg)
+O\left(\frac1{n^{10}}\right)\bigg]du .
\end{multline*}
Hence $I_0=\frac34$ and $I_1=0$.  Now let the bounds of integration
go from 0 to $\infty$ in the integral.
\par
We digress for a moment to establish notation for multiple zeta values
and alternating multiple zeta values.  The multiple zeta values are
defined by
\[
\zt(i_1,\dots,i_k)=\sum_{n_1>\cdots>n_k\ge 1}\frac1{n_1^{i_1}\cdots n_k^{i_k}}
\]
for positive integers $i_1,\dots,i_k$ with $i_1>1$.
This notation can be extended to alternating or ``colored'' multiple
zeta values by putting a bar over those exponents with an associated
sign in the numerator, as in
\[
\zt(\bar 3,\bar1,1)=\sum_{n_1>n_2>n_3\ge 1}\frac{(-1)^{n_1+n_2}}{n_1^3n_2n_3} .
\]
Note that $\zt(a_1,a_2,\dots,a_k)$ converges unless $a_1$ is an 
unbarred 1.  We have $\zt(\bar1)=-\log 2$ and
\[
\zt(\bar n)=(2^{1-n}-1)\zt(n)
\]
for $n\ge 2$.  Alternating multiple zeta values have been extensively
studied, and some identities for them are established in \cite{BBB}.
A remarkable result conjectured in \cite{BBB} but only proved thirteen
years later in \cite{Z} is
\begin{equation}
\label{zhid}
\zt(\{\bar2,1\}_n)=\frac1{8^n}\zt(\{3\}_n) ,
\end{equation}
where $\{a\}_n$ means $n$ repetitions of $a$.
\par 
Now we prove some lemmas expressing improper integrals in terms of alternating
multiple zeta values. 
\begin{lem}
\label{one}
\begin{enumerate}
\item
For integers $p\ge0$, $q\ge 1$,
\[
\int_0^\infty u^p\log(1+e^{-u})^qdu
=p!q!(-1)^q\zt(\overline{p+2},\{1\}_{q-1}).
\]
\item
For integers $p,q\ge 0$,
\[
\int_0^\infty u^p\log(1+e^{-u})^q\frac{e^{-u}}{1+e^{-u}}du
=p!q!(-1)^{q-1}\zt(\overline{p+1},\{1\}_q) .
\]
\end{enumerate}
\end{lem}
\begin{proof}
For the first part, note that
\begin{multline*}
\int_0^\infty u^p\log(1+e^{-u})^qdu=\int_0^\infty\sum_{n_1,\dots,n_q\ge 1}
\frac{(-1)^{n_1+\dots+n_q+q}u^pe^{-(n_1+\dots+n_q)u}}{n_1n_2\cdots n_q}du\\
=\sum_{n_1,\dots,n_q\ge 1}\frac{(-1)^{n_1+\dots+n_q+q}p!}
{n_1\cdots n_q(n_1+\dots+n_q)^{p+1}}\\
=\sum_{n_1,\dots,n_q\ge 1}\frac{(-1)^{n_1+\dots+n_q+q}p!q!}{n_1(n_1+n_2)\cdots 
(n_1+\dots+n_{q-1})(n_1+\dots+n_q)^{p+2}}\\
=p!q!(-1)^q\zt(\overline{p+2},\{1\}_{q-1}),
\end{multline*}
where we used \cite[Lemma 4.3]{H} in the pentultimate step.
For the second part,
\begin{multline*}
\int_0^\infty u^p\log(1+e^{-u})^q\frac{e^{-u}}{1+e^{-u}}du=\\
\int_0^\infty\sum_{n_1,\dots,n_q\ge 1}
\frac{(-1)^{n_1+\dots+n_q+q}u^pe^{-(n_1+\dots+n_q)u}}{n_1n_2\cdots n_q}
\sum_{m=1}^\infty (-1)^{m-1}e^{-mu}du\\
=\int_0^\infty\sum_{n_1,\dots,n_q,m\ge 1}
\frac{(-1)^{n_1+\dots+n_q+q+m-1}u^pe^{-(n_1+\dots+n_q+m)u}}{n_1n_2\cdots n_q}du\\
=\sum_{n_1,\dots,n_q,m\ge 1}
\frac{(-1)^{n_1+\dots+n_q+q+m-1}p!}{n_1n_2\cdots n_q(n_1+\dots+n_q+m)^{p+1}}du\\
=\sum_{n_1,\dots,n_q,m\ge 1}
\frac{(-1)^{n_1+\dots+n_q+m+q-1}p!q!}{n_1(n_1+n_2)\cdots 
(n_1+\dots+n_q)(n_1+\dots+n_q+m)^{p+1}}du\\
=p!q!(-1)^{q-1}\zt(\overline{p+1},\{1\}_q) .
\end{multline*}
\end{proof}
We note that the difference of the two parts of Lemma \ref{one} is
\begin{equation}
\label{diff}
\int_0^\infty \frac{u^p\log(1+e^{-u})^q}{1+e^{-u}}du=
p!q!(-1)^q[\zt(\overline{p+2},\{1\}_{q-1})+\zt(\overline{p+1},\{1\}_q)],
\end{equation}
which will be useful in proving the next lemma.
\begin{lem}
\label{two}
\begin{enumerate}
\item
For positive integers $p$,
\[
\int_0^\infty\frac{u^pe^{-u}}{(1+e^{-u})^2}du=-p!\zt(\bar p)
\]
\item
For positive integers $p$ and $q$,
\begin{multline*}
\int_0^\infty \frac{u^pe^{-u}\log(1+e^{-u})^q}{(1+e^{-u})^2}du=\\
p!q!\left[\sum_{k=1}^q(-1)^{k-1}(\zt(\bar p,\{1\}_k)
+\zt(\overline{p+1},\{1\}_{k-1}))-\zt(\bar p)\right] .
\end{multline*}
\end{enumerate}
\end{lem}
\begin{proof}
For the first part, we have
\begin{multline*}
\int_0^\infty\frac{u^pe^{-u}}{(1+e^{-u})^2}du=
\int_0^\infty\sum_{k=1}^\infty(-1)^{k+1}u^pke^{-ku}du=\\
\sum_{k=1}^\infty (-1)^{k+1}p!\frac{k}{k^{p+1}}=
-p!\sum_{k=1}^\infty (-1)^k\frac1{k^p}=-p!\zt(\bar p) .
\end{multline*}
For the second part, note that integration by parts gives
\begin{multline*}
\int_0^\infty \frac{u^pe^{-u}\log(1+e^{-u})^q}{(1+e^{-u})^2}du=
\int_0^\infty \frac{pu^{p-1}\log(1+e^{-u})^{q+1}}{1+e^{-u}}du\\
-\int_0^\infty \frac{u^pe^{-u}q\log(1+e^{-u})^q}{(1+e^{-u})^2}du
+\int_0^\infty \frac{u^pe^{-u}\log(1+e^{-u})^{q+1}}{(1+e^{-u})^2}du
\end{multline*}
and thus, using Eq. (\ref{diff}) with $p$ replaced by $p-1$ and $q$
replaced by $q+1$, we have
\begin{multline*}
\int_0^\infty \frac{u^pe^{-u}\log(1+e^{-u})^{q+1}}{(1+e^{-u})^2}du=
(q+1)\int_0^\infty\frac{u^pe^{-u}\log(1+e^{-u})^q}{(1+e^{-u})^2}du\\
+p!(q+1)!(-1)^q[\zt(\overline{p+1},\{1\}_q)+\zt(\bar p,\{1\}_{q+1})].
\end{multline*}
We then obtain the second part using induction on $q$ and the 
first part.
\end{proof}
\begin{lem}
\label{three}
\begin{enumerate}
\item
For positive integers $p$,
\[
\int_0^\infty\frac{u^pe^{-2u}}{(1+e^{-u})^2}du=p![\zt(\bar p)
-\zt(\overline{p+1})] .
\]
\item
For positive integers $p$ and $q$,
\[
\int_0^\infty \frac{u^pe^{-2u}\log(1+e^{-u})^q}{(1+e^{-u})^2}du=
p!q!\sum_{k=0}^q (-1)^k[\zt(\bar p,\{1\}_k)-\zt(\overline{p+1},\{1\}_k)] .
\]
\end{enumerate}
\end{lem}
\begin{proof}
For the first part, we have
\begin{multline*}
\int_0^\infty\frac{u^pe^{-2u}}{(1+e^{-u})^2}du=
\int_0^\infty\sum_{k=1}^\infty(-1)^{k+1}u^pke^{-(k+1)u}du=
\sum_{k=1}^\infty (-1)^{k+1}\frac{p!k}{(k+1)^{p+1}}\\
=\sum_{k=0}^\infty (-1)^{k+1}\frac{p!}{(k+1)^p}
-\sum_{k=0}^\infty (-1)^{k+1}\frac{p!}{(k+1)^{p+1}}
=p![\zt(\bar p)-\zt(\overline{p+1})] .
\end{multline*}
For the second  part, note that integration by parts gives
\begin{multline*}
(1+q)\int_0^\infty \frac{u^pe^{-2u}\log(1+e^{-u})^q}{(1+e^{-u})^2}du=
p\int_0^\infty\frac{u^{p-1}e^{-u}\log(1+e^{-u})^{q+1}}{1+e^{-u}}du\\
-\int_0^\infty \frac{u^pe^{-u}\log(1+e^{-u})^{q+1}}{1+e^{-u}}du
+\int_0^\infty \frac{u^pe^{-2u}\log(1+e^{-u})^{q+1}}{(1+e^{-u})^2}du ,
\end{multline*}
or, using Lemma \ref{one},
\begin{multline*}
\int_0^\infty \frac{u^pe^{-2u}\log(1+e^{-u})^{q+1}}{(1+e^{-u})^2}du=
(q+1)\int_0^\infty \frac{u^pe^{-2u}\log(1+e^{-u})^q}{(1+e^{-u})^2}du\\
-p!(q+1)!(-1)^q(\zt(\bar p,\{1\}_{q+1})-\zt(\overline{p+1},\{1\}_{q+1})).
\end{multline*}
The second part now follows by induction on $q$, using the first
part as the base case.
\end{proof}
\begin{lem}
\label{four}
\begin{enumerate}
\item
For integers $p\ge2$,
\[
\int_0^\infty \frac{u^pe^{-2u}}{(1+e^{-u})^3}du=
\frac{p!}2[\zt(\overline{p-1})-\zt(\bar p)].
\]
\item
For integers $p\ge 2$,
\[
\int_0^\infty \frac{u^pe^{-u}(1-e^{-u})}{(1+e^{-u})^3}du=-p!\zt(\overline{p-1}) .
\]
\item
For positive integers $p$,
\[
\int_0^\infty \frac{u^pe^{-u}(1-e^{-u})\log(1+e^{-u})}{(1+e^{-u})^3}du=p!
\left[\zt(\overline{p-1},1)-\frac32\zt(\overline{p-1})+\frac32\zt(\bar p)
\right].
\]
\end{enumerate}
\end{lem}
\begin{proof}
For the first part, use the result
\[
\frac{x^2}{(1+x)^3}=\frac12\sum_{k\ge 1}(-1)^kk(k-1)x^k
\]
to expand the integral as
\begin{multline*}
\frac12\sum_{k\ge 1}(-1)^kk(k-1)\int_0^\infty u^pe^{-ku}du=
\frac12\sum_{k\ge 1}(-1)^kk(k-1)\frac{p!}{k^{p+1}}=\\
\frac{p!}2\left[\sum_{k\ge 1}\frac{(-1)^k}{k^{p-1}}-\sum_{k\ge 1}\frac{(-1)^k}{k^p}
\right],
\end{multline*}
and the conclusion follows.
For the second part, use the result
\begin{equation}
\label{sq}
\frac{x-x^2}{(1+x)^3}=\sum_{k\ge 1}(-1)^{k-1}k^2x^k
\end{equation}
similarly.
For the third part, use Eq. (\ref{sq}) and expand out the logarithm
to write the integral as
\[
\int_0^\infty\sum_{k,l\ge 1}(-1)^{k+l}\frac{k^2e^{-(k+l)u}u^p}{l}du=
\sum_{k,l\ge 1}(-1)^{k+l}\frac{k^2p!}{l(k+l)^{p+1}} .
\]
Now 
\[
\frac{k^2}{l(l+k)^{p+1}}=\frac1{l(l+k)^{p-1}}-\frac2{(l+k)^p}+\frac{l}{(l+k)^{p+1}},
\]
so
\begin{multline*}
\sum_{k,l\ge 1}(-1)^{k+l}\frac{k^2p!}{l(k+l)^{p+1}}=
\sum_{k,l\ge 1}\frac{(-1)^{k+l}}{l(l+k)^{p-1}}-
2\sum_{k,l\ge 1}\frac{(-1)^{k+l}}{(l+k)^p}+
\sum_{k,l\ge 1}\frac{(-1)^{k+l}l}{(l+k)^{p+1}}\\
=\zt(\overline{p-1},1)-2(\zt(\overline{p-1})-\zt(\bar p))+
\frac12(\zt(\overline{p-1})-\zt(\bar p)),
\end{multline*}
and the conclusion follows.
\end{proof}
Using Lemma \ref{one},
\[
I_2=\frac14\int_0^\infty\log(1+e^{-u})du=-\frac14\zt(\bar2)=\frac18\zt(2)
\]
and
\begin{multline*}
I_3=\frac18\int_0^\infty[u\log(1+e^{-u})+\log(1+e^{-u})^2]du=
\frac18[-\zt(\bar3)+2\zt(\bar2,1)]=\\
\frac18\left[\frac34\zt(3)+\frac14\zt(3)\right]=\frac18\zt(3) ,
\end{multline*}
where in simplifying $I_3$ we have used the case $n=1$ of Eq. (\ref{zhid}).
\par
In all further computations, expressions for alternating multiple zeta values
are simplified using the Multiple Zeta Value Data Mine \cite{BBV}; this source
gives formulas for the alternating multiple zeta values that are
generally much simpler than those Louchard uses.
We have
\begin{multline*}
I_4=\frac18\int_0^\infty\left(-\frac{u^3e^{-u}}{6(1+e^{-u})}+
\frac{u\log(1+e^{-u})^2}2+\frac{\log(1+e^{-u})^3}3\right)du=\\
\frac18[\zt(\bar4)+\zt(\bar3,1)-2\zt(\bar2,1,1)],
\end{multline*}
and since $\zt(\bar4)=-\frac78\zt(4)$, $\zt(\bar2,1,1)=-\frac1{16}\zt(4)
+\frac12\zt(\bar3,1)$, this implies $I_4=-\frac{3}{32}\zt(4)$.
Similarly, $I_5$ is
\begin{multline*}
\frac1{96}\int_0^\infty\left(-\frac{u^4e^{-u}}{1+e^{-u}}
-\frac{2u^3e^{-u}\log(1+e^{-u})}{1+e^{-u}}+2u\log(1+e^{-u})^3
+\log(1+e^{-u})^4\right)du\\
=\frac18[2\zt(\bar5)-\zt(\bar4,1)-\zt(\bar3,1,1)+2\zt(\bar2,1,1,1)] .
\end{multline*}
Now $\zt(\bar5)=-\frac{15}{16}\zt(5)$, and from \cite{BBV}
\begin{align*}
\zt(\bar4,1)&=-\frac{29}{32}\zt(5)+\frac12\zt(2)\zt(3)\\
\zt(\bar2,1,1,1)&=\frac{31}{64}\zt(5)-\frac14\zt(2)\zt(3)
+\frac12\zt(\bar3,1,1),
\end{align*}
giving the result $I_5=-\frac18\zt(2)\zt(3)$.  These results corroborate
Louchard's.
\par
Now we can go further.  Using Lemmas \ref{one} and \ref{two},
\begin{multline*}
I_6=\frac1{5760}\int_0^\infty\bigg(\frac{5u^6e^{-u}}{(1+e^{-u})^2}
-\frac{18u^5e^{-u}}{1+e^{-u}}
-\frac{60u^4e^{-u}\log(1+e^{-u})}{1+e^{-u}}\\
-\frac{60u^3e^{-u}\log(1+e^{-u})^2}{1+e^{-u}}
+30u\log(1+e^{-u})^4+12\log(1+e^{-u})^5\bigg)du\\
=\frac18[-2\zt(\bar6)-2\zt(\bar5,1)+\zt(\bar4,1,1)+
\zt(\bar3,1,1,1)-2\zt(\bar2,\{1\}_4)],
\end{multline*}
and since $\zt(\bar6)=-\frac{31}{32}\zt(6)$, and
\begin{align*}
\zt(\bar4,1,1)&=\frac5{16}\zt(6)-\frac14\zt(3)^2+\frac32\zt(\bar5,1)\\
\zt(\bar2,\{1\}_4)&=-\frac{11}{64}\zt(6)+\frac18\zt(3)^2-\frac14\zt(\bar5,1)
+\frac12\zt(\bar3,1,1,1)
\end{align*}
we obtain $I_6=\frac18\left[\frac{83}{32}\zt(6)-\frac12\zt(3)^2\right]$.
Similarly, using in addition Lemma \ref{three}, we have
\begin{multline*}
I_7=\frac1{11520}\int_0^\infty\bigg(\frac{10u^6e^{-2u}}{(1+e^{-u})^2}
+\frac{5u^7e^{-u}}{(1+e^{-u})^2}+\frac{10u^6e^{-u}\log(1+e^{-u})}{(1+e^{-u})^2}
-\frac{18u^6e^{-u}}{1+e^{-u}}\\
-\frac{36u^5e^{-u}\log(1+e^{-u})}{1+e^{-u}}
-\frac{60u^4e^{-u}\log(1+e^{-u})^2}{1+e^{-u}}
-\frac{40u^3e^{-u}\log(1+e^{-u})^3}{1+e^{-u}}\\
+12u\log(1+e^{-u})^5+4\log(1+e^{-u})^6\bigg)du\\
=\frac18\left[-\frac{17}2\zt(\bar7)+2\zt(\bar6,1)+2\zt(\bar5,1,1)-\zt(\bar4,1,1,1)
-\zt(\bar3,\{1\}_4)+2\zt(\bar2,\{1\}_5)\right].
\end{multline*}
From this and $\zt(\bar7)=-\frac{63}{64}\zt(7)$, together with
\begin{align*}
\zt(\bar6,1)&=-\frac{251}{128}\zt(7)+\frac78\zt(3)\zt(4)+\frac12\zt(2)\zt(5)\\
\zt(\bar4,1,1,1)&=\frac{315}{128}\zt(7)-\frac54\zt(3)\zt(4)
-\frac12\zt(2)\zt(5)+\frac32\zt(\bar5,1,1)\\
\zt(\bar2,\{1\}_5)&=-\frac{31}{128}\zt(7)+\frac3{16}\zt(3)\zt(4)
-\frac14\zt(\bar5,1,1)+\frac12\zt(\bar3,\{1\}_4),
\end{align*}
it follows that 
$I_7=\frac18\left[\frac32\zt(7)+\frac{27}8\zt(3)\zt(4)+\frac32\zt(2)\zt(5)
\right]$.
\par
For the next two $I_n$ we need all four lemmas.  First
\begin{multline*}
I_8=\frac1{1451520}\int_0^\infty\bigg(
-\frac{35u^9e^{-u}(1-e^{-u})}{(1+e^{-u})^3}
+\frac{378u^8e^{-u}}{(1+e^{-u})^2}
+\frac{630u^7e^{-u}\log(1+e^{-u})}{(1+e^{-u})^2}\\
+\frac{630u^6e^{-u}\log(1+e^{-u})^2}{(1+e^{-u})^2}
+\frac{630u^7e^{-2u}}{(1+e^{-u})^2}
+\frac{1260u^6e^{-2u}\log(1+e^{-u})}{(1+e^{-u})^2}
-\frac{810u^7e^{-u}}{1+e^{-u}}\\
-\frac{2268u^6e^{-u}\log(1+e^{-u})}{1+e^{-u}}
-\frac{2268u^5e^{-u}\log(1+e^{-u})^2}{1+e^{-u}}
-\frac{2520u^4e^{-u}\log(1+e^{-u})^3}{1+e^{-u}}\\
-\frac{1260u^3e^{-u}\log(1+e^{-u})^4}{1+e^{-u}}
+252u\log(1+e^{-u})^6+72\log(1+e^{-u})^7\bigg)du\\
=\frac18\bigg[\frac{17}2\zt(\bar 8)+\frac{17}2\zt(\bar7,1)-2\zt(\bar6,1,1)
-2\zt(\bar5,1,1,1)+\zt(\bar4,\{1\}_4)\\
+\zt(\bar3,\{1\}_5)-2\zt(\bar2,\{1\}_6)\bigg],
\end{multline*}
from which, using $\zt(\bar8)=-\frac{127}{128}\zt(8)$ together with
\begin{align*}
\zt(\bar6,1,1)&=\frac{917}{768}\zt(8)-\frac12\zt(3)\zt(5)-\frac14\zt(2)\zt(3)^2
+\frac52\zt(\bar7,1)\\
\zt(\bar4,\{1\}_4)&=-\frac{343}{192}\zt(8)+\frac12\zt(3)\zt(5)
+\frac12\zt(2)\zt(3)^2-\frac52\zt(\bar7,1)+\frac32\zt(\bar5,1,1,1)\\
\zt(\bar2,\{1\}_6)&=\frac{449}{1536}\zt(8)-\frac18\zt(2)\zt(3)^2
+\frac12(\bar7,1)-\frac14\zt(\bar5,1,1,1)+\frac12\zt(\bar3,\{1\}_5),
\end{align*}
it follows that $I_8=\frac18\bigg[-\frac{2533}{192}\zt(8)+\frac32\zt(3)\zt(5)
+\frac54\zt(2)\zt(3)^2\bigg]$.
\par
Finally, $I_9$ is
\begin{multline*}
\frac1{2903040}\int_0^\infty\bigg(-\frac{35u^{10}e^{-u}(1-e^{-u})}{(1+e^{-u})^3}
-\frac{70u^9e^{-u}(1-e^{-u})\log(1+e^{-u})}{(1+e^{-u})^3}\\
-\frac{210u^9e^{-2u}}{(1+e^{-u})^3}
+\frac{756u^8e^{-2u}}{(1+e^{-u})^2}
+\frac{1260u^7e^{-2u}\log(1+e^{-u})}{(1+e^{-u})^2}
+\frac{1260u^6e^{-2u}\log(1+e^{-u})^2}{(1+e^{-u})^2}\\
+\frac{378u^9e^{-u}}{(1+e^{-u})^2}
+\frac{756u^8e^{-u}\log(1+e^{-u})}{(1+e^{-u})^2}
+\frac{630u^7e^{-u}\log(1+e^{-u})^2}{(1+e^{-u})^2}\\
+\frac{420u^6e^{-u}\log(1+e^{-u})^3}{(1+e^{-u})^2}
-\frac{810u^8e^{-u}}{1+e^{-u}}
-\frac{1620u^7e^{-u}\log(1+e^{-u})}{1+e^{-u}}\\
-\frac{2268u^6e^{-u}\log(1+e^{-u})^2}{1+e^{-u}}
-\frac{1512u^5e^{-u}\log(1+e^{-u})^3}{1+e^{-u}}
-\frac{1260u^4e^{-u}\log(1+e^{-u})^4}{1+e^{-u}}\\
-\frac{504u^3e^{-u}\log(1+e^{-u})^5}{1+e^{-u}}
+72u\log(1+e^{-u})^7+18\log(1+e^{-u})^8\bigg)du\\
=\frac18\bigg[62\zt(\bar 9)-\frac{17}2\zt(\bar 8,1)
-\frac{17}2\zt(\bar 7,1,1)+2\zt(\bar 6,1,1,1)+2\zt(\bar 5,\{1\}_4)\\
-\zt(\bar 4,\{1\}_5)-\zt(\bar 3,\{1\}_6)+2\zt(\bar 2,\{1\}_7)
\bigg] .
\end{multline*}
Together with $\zt(\bar9)=-\frac{255}{256}\zt(9)$ and
\begin{align*}
\zt(\bar8,1)&=-\frac{1529}{512}\zt(9)+\frac{31}{32}\zt(3)\zt(6)+\frac78
\zt(4)\zt(5)+\frac12\zt(2)\zt(7)\\
\zt(\bar6,1,1,1)&=\frac{24983}{3072}\zt(9)-\frac{175}{64}\zt(3)\zt(6)
-\frac{41}{16}\zt(4)\zt(6)-\frac54\zt(2)\zt(7)+\frac1{12}\zt(3)^3\\
&+\frac52\zt(\bar 7,1,1)\\
\zt(\bar4,\{1\}_5)&=-\frac{14273}{1536}\zt(9)+\frac{49}{16}\zt(3)\zt(6)
+3\zt(4)\zt(5)+\frac32\zt(2)\zt(7)-\frac16\zt(3)^3\\
&-\frac52\zt(\bar7,1,1)+\frac32\zt(\bar5,\{1\}_4)\\
\zt(\bar2,\{1\}_7)&=\frac{12749}{6144}\zt(9)-\frac{83}{128}\zt(3)\zt(6)
-\frac{21}{32}\zt(4)\zt(5)-\frac38\zt(2)\zt(7)+\frac1{24}\zt(3)^3\\
&+\frac12\zt(\bar7,1,1)-\frac14\zt(\bar5,\{1\}_4)+\frac12\zt(\bar3,\{1\}_6)
\end{align*}
this gives
\[
I_9=\frac18\left[-\frac{20}3\zt(9)-\frac{289}{16}\zt(3)\zt(6)-\frac{135}8
\zt(4)\zt(5)-9\zt(2)\zt(7)+\frac5{12}\zt(3)^3\right] .
\]
We make two conjectures.
\begin{conj}
There is a sequence $a_0,a_1,a_2,\dots$ of rational numbers 
that begins $-2,1,-2,\frac{17}2,-62,\dots$ such that,
for $n\ge 2$,
\[
I_n=\frac18\sum_{j=2}^n(-1)^na_{\lfl\frac{j-1}2\rfl}
\zt(\bar j,\{1\}_{n-j}) .
\]
\end{conj}
\begin{conj} 
$I_n$ is a rational polynomial in the ordinary zeta values
$\zt(i)$, $i>2$.
\end{conj}

\end{document}